\documentclass[12pt]{amsart}
\usepackage[colorlinks=true]{hyperref}
\usepackage{amsmath,amssymb,amsthm,amsxtra,mathrsfs}
\usepackage{graphicx}
\usepackage{color}
\usepackage{booktabs} 
\usepackage{paralist}
\usepackage{tikz} 
\usepackage{pdflscape}

\numberwithin{equation}{section}

\theoremstyle{plain}
\newtheorem{thm}{Theorem}[section]
\newtheorem{prop}[thm]{Proposition}
\newtheorem{lem}[thm]{Lemma}

\newtheorem{conj}[thm]{Conjecture}
\newtheorem{defn}[thm]{Definition}
\newtheorem*{cor*}{Corollary}
\newtheorem*{prop*}{Proposition}
\newtheorem*{thm*}{Theorem}

\theoremstyle{remark}

\newtheorem{rmk}[equation]{Remark}

\DeclareMathOperator{\GL}{GL}

\newcommand{\Lfunction}{\hbox{$L$-function}}

\newcommand{\A}{\mathbb A}
\newcommand{\C}{\mathbb C}

\newcommand{\Q}{\mathbb Q}
\newcommand{\R}{\mathbb R}
\newcommand{\Z}{\mathbb Z}

\renewcommand{\Re}{\mathrm{Re}}

\def\term#1{\textbf{\textit{#1}}}

\begin{document}

\title[Analytic $L$-functions]{Analytic $L$-functions: Definitions, Theorems, and Connections}
\author[Farmer]{David W. Farmer}
\address{American Institute of Mathematics, San Jose, CA USA}
\author[Pitale]{Ameya Pitale}
\address{University of Oklahoma, Norman, OK USA}
\author[Ryan]{Nathan C. Ryan}
\address{Bucknell University, Lewisburg, PA USA}
\author[Schmidt]{Ralf Schmidt}
\address{University of Oklahoma, Norman, OK USA}

\date{\today}

\begin{abstract}
$L$-functions can be viewed axiomatically, such as in the formulation due to Selberg,
or they can be seen as arising from cuspidal automorphic representations of $\GL(n)$,
as first described by Langlands.
Conjecturally these two descriptions of $L$-functions are
the same, but it is not even clear that these are describing the same
set of objects.  We propose a collection of axioms that bridges the
gap between the very general analytic axioms due to
Selberg and the very particular and algebraic construction
due to Langlands.  Along the way we prove theorems about $L$-functions
that satisfy our axioms and state conjectures that arise naturally
from our axioms.
\end{abstract}

\maketitle

\tableofcontents

\section{Introduction}

Different mathematicians mean different things when they say
``$L$-function.''  Some mean an element of the Selberg class and others
might mean a Dirichlet series with Euler product and require that it
be associated
to an automorphic form.  For some people an $L$-function has to be
entire, for others it can have poles on the edge of the critical
strip, for yet 
others it can even have poles in other locations.

In this paper we show how one can attach adjectives to $L$-functions (and which
adjectives one should attach, as determined by one's goals)
in such a way that the resulting classes
of $L$-functions provide a detailed framework to understanding $L$-functions.
This framework can be used to clarify
the distinctions between various classes, and also to unify by
showing connections between them.  In
Section~\ref{sec:types}, we define the following sets of $L$-functions
(see that Section for definitions):
\begin{itemize}
\item Tempered balanced analytic primitive entire $L$-functions.  These $L$-functions are
  defined axiomatically, with precise restrictions on their
  functional equation and Euler product.
\item $\Q$-automorphic $L$-functions.  These $L$-functions are
  associated to tempered balanced unitary cuspidal automorphic representations of $\GL(n,\A_\Q)$.
\end{itemize}
The above sets are believed to contain all primitive $L$-functions that
are expected to satisfy analogues of the Riemann Hypothesis,
and conjecturally the two sets are essentially equal.

Within each of the above sets are distinguished subsets which,
conjecturally, contain all $L$-functions arising from arithmetic
objects.
\begin{itemize}
\item  $L$-functions of algebraic type (analytic and $\Q$-automorphic).  These are characterized
by conditions on the $\Gamma$-factors in the functional equation.
\item  $L$-functions of arithmetic type (analytic and $\Q$-automorphic).  These are characterized
by conditions on the coefficients in the Dirichlet series.
\end{itemize}
Conjecturally all four of these sets of $L$-functions are equal and arise from the following arithmetic objects:
\begin{itemize}
\item Pure motives.
\item Geometric Galois representations.
\end{itemize}
Associated to each such arithmetic object is an $L$-function.  Conjecturally
those sets of $L$-functions are equal, and coincide with the
four subsets of $L$-functions mentioned previously.
The conjectured relationships between these sets of $L$-functions is
shown in Figure~\ref{fig:diagram1}.

In Figure~\ref{fig:diagram2} in Section~\ref{sec:connections} we discuss in more detail the conjectured relationships between the 
sets of $L$-functions described above.  Precise descriptions of
each of these sets are given in Sections~\ref{sec:types} and~\ref{sec:aa}.

\usetikzlibrary{arrows,
                calc,chains,
                decorations.pathreplacing,decorations.pathmorphing,
                fit,
                positioning}

\begin{figure}

\tikzset{
  line/.style={draw, thick},
        }

\begin{tikzpicture}[node distance = 4mm, start chain = A going below, font = \sffamily, > = stealth', PC/.style = {
        rectangle, rounded corners,
        draw=black, very thick,
        text width=8em,
        minimum height=1.5em,
        align=center,
        on chain}, 
blank/.style={
        minimum height=1.5em,
        align=center, on chain},
                    ]

\node[PC]   {Tempered $\mathbb{Q}$-automorphic \mbox{$L$-functions} of algebraic type}; 
\node[PC]   {Primitive analytic $L$-functions of algebraic type};
\node[blank,right=of A-1,xshift=-.375cm] {$\subsetneq$};
\node[blank,right=of A-2,xshift=-.375cm] {$\subsetneq$};
\node[PC,right=of A-3,xshift=-.375cm] {Tempered $\mathbb{Q}$-automorphic $L$-functions};
\node[PC,right=of A-4,xshift=-.375cm] {Primitive analytic $L$-functions};

\node[PC,left=of A-1,yshift=10mm] {$L$-functions of pure irreducible motives};

\node[PC,left=of A-1,yshift=-12mm] {$L$-functions of geometric Galois representations};

\draw[shorten <=1mm,shorten >=1mm] (A-2.north)  edge[double distance=2pt]  (A-1.south);

\draw[shorten <=1mm,shorten >=1mm] (A-6.north)  edge[double distance=2pt]  (A-5.south); 

\draw[shorten <=1mm,shorten >=1mm] (A-1.west)  edge[double distance=2pt]  (A-7.south east); 

\draw[shorten <=1mm,shorten >=1mm] (A-1.west)  edge[double distance=2pt]  (A-8.north east);

\draw[shorten <=1mm,shorten >=1mm] (A-8.north)  edge[double distance=2pt]  (A-7.south); 

\end{tikzpicture}
\caption{Conjectured relationships between the sets of $L$-functions considered in this paper.}\label{fig:diagram1}
\end{figure}
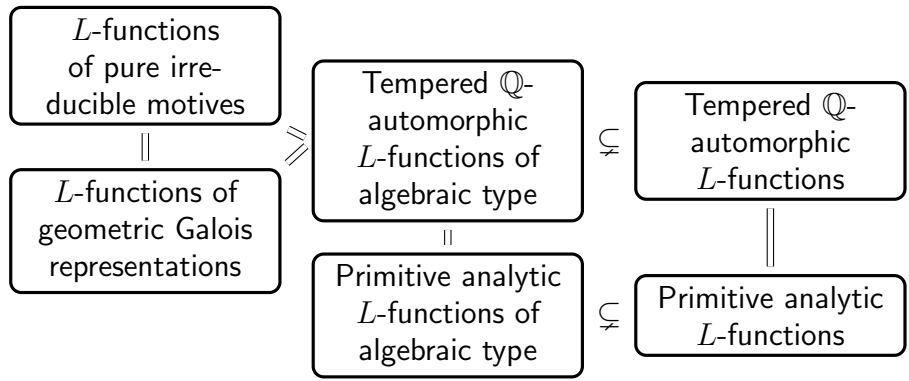

\section{Two views of \texorpdfstring{$L$}{}-functions}\label{sec:types}

\subsection{Analytic \texorpdfstring{$L$}{}-functions}\label{sec:axioms}
\label{sec:defn_analytic}
The first set of $L$-functions in our discussion is defined
axiomatically.  

Throughout the axioms, $s=\sigma+ it$ is a complex variable with $\sigma$
and $t$ real.

A \term{tempered balanced analytic $L$-function} is a function $L(s)$ which satisfies the five axioms below. In this paper we will also refer to `tempered analytic $L$-functions' and
`analytic $L$-functions,' which are obtained by relaxing some of these axioms.
See Section~\ref{sec:axioms_comments} for a discussion.

\begin{description}
\setlength{\itemindent}{-\leftmargin}
\setlength{\listparindent}{\parindent}
\item[Axiom $1$ (Analytic properties)]
$L(s)$ is given by a Dirichlet series
\index{Dirichlet series}
\[
\label{eqn:DS}\tag{Ax1.1}
L(s)=\sum_{n=1}^\infty \frac{a_n}{n^s},
\]
where $a_n\in \C$.

{\renewcommand{\labelenumi}{\alph{enumi})}
\begin{enumerate}

\smallskip

\item \emph{Convergence:} $L(s)$ converges absolutely for $\sigma>1$.

\smallskip

\item \emph{Analytic continuation:} $L(s)$ continues
to a meromorphic function having only finitely many poles,
with all poles 
lying on the $\sigma=1$ line.
\end{enumerate}
}

\medskip

\item[Axiom $2$ (Functional equation)]
\index{functional equation!axioms}
There is a positive integer $N$ \linebreak called the \term{conductor} of the \Lfunction,
\index{conductor}
a positive integer $d$ called the \term{degree} of the \Lfunction,
\index{degree}
a pair of non-negative integers $(d_1, d_2)$ called the \term{signature} of the \Lfunction,
\index{signature}
where $d=d_1+2d_2$,
and complex numbers $\{\mu_j\}_{j=1}^{d_1}$ and $\{\nu_k\}_{k=1}^{d_2}$
called the \term{spectral parameters} of the \Lfunction,
\index{spectral parameters}
such that the \term{completed \Lfunction} 
\index{completed \Lfunction}\index{$L$-function@\Lfunction!completed}
\[
\label{eqn:Lambda}\tag{Ax2.1}
\Lambda(s) =\mathstrut  N^{s/2} 
\prod_{j=1}^{d_1} \Gamma_\R(s+ \mu_j)
\prod_{k=1}^{d_2} \Gamma_\C(s+ \nu_k)
\cdot L(s)
\]
has  the following properties:

{\renewcommand{\labelenumi}{\alph{enumi})}
\begin{enumerate}

\smallskip

\item \emph{Bounded in vertical strips:}
Away from the poles of the \Lfunction,
$\Lambda(s)$ is bounded in vertical strips $\sigma_1 \leq \sigma \leq \sigma_2$.

\smallskip

\item \emph{Functional equation:} There exists $\varepsilon\in \C$,
\index{functional equation!general}
called the \term{sign} of the functional equation, such that
\index{sign!of the functional equation}
\[
\label{eqn:FE}\tag{Ax2.2}
\Lambda(s) 
=\mathstrut  \varepsilon \overline{\Lambda}(1-s).
\]

%


\end{enumerate}
}

\medskip
\item[Axiom $3$ (Euler product)] There is a product formula
\index{Euler product}
\[
\label{eqn:EP}\tag{Ax3.1}
L(s)= \prod_{p \, {\rm prime}} F_p(p^{-s})^{-1},
\]
absolutely convergent for $\sigma > 1$.

{
\renewcommand{\labelenumi}{\alph{enumi})}
\begin{enumerate}

\smallskip

\item \emph{Polynomial:} $F_p$ is a polynomial with $F_p(0)=1$.

\item \emph{Degree:} Let $d_p$ be the degree of $F_p$.  If $p\nmid N$ then $d_p=d$,
and if $p\mid N$ then $d_p<d$.

\smallskip

%
%
%


\end{enumerate}
}

\medskip
\item[Axiom $4$ (Temperedness)]   The spectral parameters and Satake parameters satisfy precise bounds.
{
\renewcommand{\labelenumi}{\alph{enumi})}
\begin{enumerate}

\smallskip

\item \emph{Selberg bound:} \label{axiom:preciseselbergeigenvalue}
For every $j$ we have
$\Re(\mu_j) \in \{0,1\}$ and
$\Re(\nu_k) \in \{\frac12,1,\frac32,2,...\}$.

\item \emph{Ramanujan bound:} Write $F_p$ in factored form as
\[
\label{eqn:ram-bound}\tag{Ax4.1}
 F_p(z) = (1-\alpha_{1,p} z)\cdots (1-\alpha_{d_p,p} z)
 \]
 with $\alpha_{j,p} \not = 0$. If $p\nmid N$ then $|\alpha_{j,p}| = 1$ for all $j$. If $p\mid N$ then $|\alpha_{j,p}|= p^{-m_j/2}$ for some $m_j\in \{0,1,2,...\}$, and  $\sum m_j \le d-d_p$.
\end{enumerate}
}
\medskip
\item[Axiom $5$ (Central character)] There exists a Dirichlet character $\chi$ mod~$N$,
\index{central character!of an \Lfunction}
called the \term{central character} of the \Lfunction.
{
\renewcommand{\labelenumi}{\alph{enumi})}
\begin{enumerate}

\smallskip

\item \emph{Highest degree term:} For every prime $p$,
\[
\label{eqn:Fpchi}\tag{Ax5.1}
F_p(z)=1-a_p z + \cdots + (-1)^d\chi(p) z^d .
\]

\item \emph{Balanced:} We have
$\mathrm{Im} \left(\sum \mu_j +  \sum(2\nu_k+1)\right) =0$.

\item \emph{Parity:} The spectral parameters determine the
parity of the central character:
\[
\label{eqn:parityaxiom}\tag{Ax5.2}
\chi(-1) = (-1)^{\sum \mu_j +  \sum(2\nu_k+1)}.
\]
\end{enumerate}
}

\end{description}

\subsubsection{Comments on the terminology}
\label{sec:axioms_comments}

The term \term{balanced} is described by Axiom 5b).  
The summand ``$+1$'' can obviously be omitted from the condition,
but we include it for uniformity with Axiom~5c).
Note that Axiom~5c) would be problematic if it were not
assumed that the exponent on $-1$ was an integer.
If we omit the modifier `balanced' when describing an $L$-function,
then we mean a function of the form $L(s+i y)$ where 
$L(s)$ is balanced and $y\in \R$.
If $L(s)$ is a (not necessarily balanced) $L$-function, then it is
straightforward to check that there exists exactly one $y_0\in \R$
such that $L(s+i y_0)$ is balanced.

The term \term{tempered} refers to both the Selberg bound (Axiom 4a) and
the Ramanujan bound (Axiom 4b).  Neither bound has been proven for
most automorphic $L$-functions, but if those axioms fail
for an automorphic $L$-function, they must fail
in a specific way arising from the fact that the underlying representation
is unitary.  A precise description of the possibilities is given by
the \term{unitary pairing condition}, described in the appendix.

In the functional equation \eqref{eqn:FE}, the function $\overline{\Lambda}$ is the
Schwartz reflection of $\Lambda$, defined for arbitrary analytic
functions $f$ by $\overline{f}(z) = \overline{f(\overline{z})}$.
The tuple $(\varepsilon,N,\{\mu_1,\ldots,\mu_{d_1}\},\{\nu_k,\ldots,\nu_{d_2}\})$
is the \term{functional equation data} of the $L$-function.  The sign $\varepsilon$ of the functional equation has absolute value 1: to see this, apply the
functional equation twice to get $\Lambda(s) = \varepsilon\overline{\varepsilon}\Lambda(s)$.

In the Euler product, the polynomials $F_p$ are known as the
\term{local factors}, and the reciprocal roots $\alpha_{j,p}$ are
called the \term{Satake parameters} at $p$. 
If $p\mid N$ then we say $p$ is a \term{bad} prime,
and if $p\nmid N$ then $p$ is \term{good}.

It follows straight from the definition that if $L_1(s)$ and $L_2(s)$ are
analytic $L$-functions then so is $L_1(s) L_2(s)$.
And if both $L_1$ and $L_2$ are balanced, or tempered, then so is their product.
If the analytic $L$-function $L(s)$ cannot be written non-trivially as
$L(s)=L_1(s) L_2(s)$, then we say
that $L$ is \term{primitive}.  Here `non-trivially' refers to the fact
that the constant function 1 is a degree 0 $L$-function.
It follows from the Selberg orthogonality
conjecture~\cite{sel, cg} 
that $L$-functions factor uniquely into
primitive $L$-functions:

\begin{thm}\cite{cg}\label{thm:factorization} Assume the Selberg orthogonality
conjecture. If $L(s)$ is an analytic $L$-function then
\begin{equation}
L(s) = L_1(s) \cdots L_n(s)
\end{equation}
where each $L_j$ is a nontrivial primitive analytic $L$-function.
The representation is unique except for the order of the
factors.
\end{thm}

If `tempered' is included as a condition in Theorem~\ref{thm:factorization}, then each
$L_j$ is tempered.  If `balanced' is included as a condition, then
the conclusion can be written as
\begin{equation}
L(s) = L_1(s+ i y_1) \cdots L_n(s+i y_n)
\end{equation}
where each $L_j$ is balanced, $y_j\in \R$, and
$\sum d_j y_j = 0$, where $d_j$ is the degree of $L_j$.
Examples such as $\zeta(s+5i)\zeta(s-5i)$ and $L(s+6i,\chi)L(s-3i, E)$
show that a non-primitive balanced $L$-function cannot necessarily be written as
a product
of primitive balanced $L$-functions.

The motivation for the Central Character axioms comes from
the $\Q$-automorphic $L$-functions which we describe in the next section.
See the discussion preceding equations~\eqref{eqn:Fpchi2} and~\eqref{eqn:parityaxiom2}:  note, in particular, that discussion explains why Axioms~5a) and~5b) are equivalent
if the $L$-function is $\Q$-automorphic.

%


\subsection{\texorpdfstring{$\Q$}{}-automorphic \texorpdfstring{$L$}{}-functions}

For a number field $F$, let $\A_F$ denote the ring of adeles of $F$. In this section we consider $L$-functions of cuspidal automorphic representations $\pi$ of the group $\GL(n,\A_\Q)$. For such $\pi$ we will always use the same conventions as in \cite{Cogdell04}; in particular, we assume $\pi$ to be irreducible and unitary. Then $\pi$ admits a unitary central character $\omega_\pi$, which is a character of 
\[
 \Q^\times \backslash \A_\Q^\times = \R_{>0} \times \prod_{p< \infty} \Z_p^\times.
\]
There exists a unique real number $y$ and a character $\chi$ of
$\Q^\times\backslash\A_\Q^\times$ of finite order such that
$\omega_\pi=|\cdot|^{iy}\chi$. The character $\chi$ corresponds to
a Dirichlet character, also denoted by $\chi$:
\begin{equation}\label{ccexpleq2}
 \omega_\pi(x)=x^{iy}\qquad\text{for }x\in\R_{>0}.
\end{equation}
We say $\pi$ is \term{balanced} if the restriction of $\omega_\pi$ to $\R_{>0}$ is trivial,
that is, if $y=0$.
Evidently, this is equivalent to $\omega_\pi$ being of finite order. In this case $\omega_\pi$ corresponds to a Dirichlet character $\chi$. The correspondence is such that if $\omega_\pi$ factors as $\prod_{p\leq\infty}\omega_{\pi,p}$, then
\begin{equation}\label{eqn:omegap}
 \omega_{\pi,p}(p)=\chi(p)
\end{equation}
and
\begin{equation}\label{eqn:omegainfty}
 \omega_{\pi,\infty}(-1)=\chi(-1).
\end{equation}

\begin{defn}\label{def:automorphic}
 Let $\pi=\otimes_{p\leq\infty}\pi_p$ be a cuspidal automorphic representation of $\GL(d,\A_\Q)$.  Let 
 \[
  L(s,\pi)= \prod_{p < \infty} L(s,\pi_p) 
 \]
 be the finite part of the Langlands $L$-function associated to $\pi$ with respect to the standard representation of the dual group $\GL(d,\C)$.  We call $L(s,\pi)$ a \term{$\Q$-automorphic $L$-function}. 
\end{defn}

We do not consider automorphic representations for groups other then $\GL(n)$ or for number fields other than $\Q$. This is not a serious restriction, since more general automorphic representations can always be transferred to $\GL(n)$ over $\Q$, at least conjecturally; see \cite{Arthur} for some recent and deep results. This transfer may not be cuspidal, however, so our definition \emph{will} exclude some more general automorphic $L$-functions. The following examples will illustrate why restricting to $\GL(n)$ over $\Q$ is not harmful, and in fact desirable for our purposes.

First consider Hilbert modular cuspforms on a real quadratic number field $F$. As explained in \cite{RaghuramTanabe2011}, such modular forms correspond to cuspidal automorphic representations $\pi$ of $\GL(2,\A_F)$. The $L$-function $L(s,\pi)$ is a degree $2$ $L$-function \emph{over $F$}. We may consider the  automorphic induction $\mathcal{AI}_{F/\Q}(\pi)$, which is an automorphic representation of $\GL(4,\A_\Q)$. It has the same $L$-function $L(s,\pi)$, but now we consider it as a degree $4$ $L$-function \emph{over $\Q$}. If $\mathcal{AI}_{F/\Q}(\pi)$ is cuspidal, then $L(s,\pi)$ is included in our definition of automorphic $L$-function. Assume that $\mathcal{AI}_{F/\Q}(\pi)$ is not cuspidal. Then, by \textsection~3.6 of \cite{ArthurClozel1989}, $\pi$ is Galois-invariant, and therefore in the image of the base change map from $\GL(2,\A_\Q)$ to $\GL(2,\A_F)$. It follows that there exists a cuspidal, automorphic representation $\pi_\Q$ of $\GL(2,\A_\Q)$ such that $L(s,\pi)=L(s,\pi_\Q)L(s,\pi_\Q\otimes\chi_{F/\Q})$, where $\chi_{F/\Q}$ is the quadratic character corresponding to the extension $F/\Q$. We will later compare the class of $L$-functions according to Definition \ref{def:automorphic} with the class of \emph{primitive} analytic $L$-functions. It is therefore advantageous to exclude a non-primitive example like $L(s,\pi)$ from our definition of automorphic $L$-function.

As our second example, consider Siegel modular cuspforms $F$ of degree $2$. Such $F$ correspond to cuspidal automorphic representations $\pi$ of ${\rm GSp}(4,\A_\Q)$, which, at least conjecturally, can be transferred to $\GL(4,\A_\Q)$. For example, if $F$ has full level and is not a Saito-Kurokawa lifting, then the transfer is established in \cite{PitaleSahaSchmidt2014}. This transfer is again cuspidal, and thus the spin (degree $4$) $L$-function $L(s,\pi)$ is included in Definition \ref{def:automorphic}. Assume however that $F$ is a Saito-Kurokawa lifting. Then the transfer to $\GL(4,\A_\Q)$ still exists, but is no longer cuspidal. Hence, in this case, $L(s,\pi)$ is not included in our definition of automorphic $L$-function. This is desirable, since $L(s,\pi)$ 
is of the form $\zeta(s-\frac12)\zeta(s+\frac12)L(s, f)$, which is neither primitive nor satisfies the Ramanujan condition. 
Thus, it should be excluded from our comparison with the class of primitive tempered analytic $L$-functions.

We see from these examples that $L$-functions of cuspidal automorphic representations are in general
not primitive, in the sense that they may factor as products of $L$-functions of smaller degrees.
The following lemma shows that this cannot happen for $\Q$-automorphic $L$-functions.
\begin{lem}\label{Qautoprimitivelemma}
 Let $\pi$ be a cuspidal automorphic representation of \linebreak$\GL(d,\A_\Q)$. Then the $\Q$-automorphic $L$-function $L(s,\pi)$ is primitive, i.e., if
 \begin{equation}\label{Qautoprimitivelemmaeq1}
  L(s,\pi)=\prod_{j=1}^mL(s,\pi_j)
 \end{equation}
 with cuspidal automorphic representations $\pi_j$ of $\GL(d_j,\A_\Q)$ ($d_j>0$) for $1\leq j\leq
m$, then $m=1$.
\end{lem}
\begin{proof}
We consider partial $L$-functions and twist \eqref{Qautoprimitivelemmaeq1} by the contragredient
$\pi_1^\vee$ of $\pi_1$:
\begin{equation}\label{Qautoprimitivelemmaeq2}
  L_S(s,\pi\times\pi_1^\vee)=\prod_{j=1}^mL_S(s,\pi_j\times\pi_1^\vee).
\end{equation}
By the Corollary to Theorem 2.4 of \cite{Cogdell04}, $L_S(s,\pi_1\times\pi_1^\vee)$ has a pole at
$s=1$. By Theorem 5.2 of \cite{Shahidi1981}, $L_S(s,\pi_j\times\pi_1^\vee)$ has no zeros on ${\rm
Re}(s)=1$ for any $j$. It follows that the right hand side of \eqref{Qautoprimitivelemmaeq2} has a
pole at $s=1$. Hence so does the left hand side, which again by Theorem 2.4 of \cite{Cogdell04}
implies that $\pi=\pi_1$. In other words, we must have $m=1$.
\end{proof}

\section{\texorpdfstring{$\Q$}{}-automorphic \texorpdfstring{$L$}{}-functions and the axioms}\label{sec:Q-auto}

Now that we have defined $\Q$-automorphic $L$-functions and have identified a collection of axioms for analytic $L$-functions,
we begin to show that $\Q$-automorphic $L$-functions satisfy the axioms.

\begin{thm}\label{thm:niceness}
 Let $L(s,\pi)$ be a $\Q$-automorphic $L$-function.
There is a positive integer $N$, a pair of non-negative integers $(d_1, d_2)$ so that $d_1+2d_2=d$,
and complex numbers $\{\mu_j\}$ and $\{\nu_j\}$ such that 
\begin{equation}\label{thm:nicenesseq0}
 \Lambda(s,\pi) =\mathstrut  N^{s/2} 
 \prod_{j=1}^{d_1} \Gamma_\R(s+ \mu_j)
 \prod_{j=1}^{d_2} \Gamma_\C(s+ \nu_j)
 \cdot L(s,\pi)
\end{equation}
has  the following properties:
\begin{enumerate}
 \item \label{item:ent} $\Lambda(s,\pi)$ is entire.
 \item \label{item:bdd} $\Lambda(s,\pi)$ is bounded in vertical strips $\sigma_1 \leq \sigma \leq \sigma_2$.
 \item\label{item:FE} There exists $\varepsilon\in \C$ such that
  \begin{equation*}
   \Lambda(s,\pi)=\varepsilon \overline{\Lambda}(1-s,\pi).
  \end{equation*}
 \item \label{item:bal} If $\pi$ is balanced, then $L(s,\pi)$ is balanced.
\end{enumerate}
\end{thm}

In other words, $\Q$-automorphic $L$-functions satisfy Axioms 1) and 2) as described in Section~\ref{sec:axioms}.

\begin{proof} For items \eqref{item:ent}, \eqref{item:bdd} and \eqref{item:FE}, see \cite[Theorems 2.3 and 2.4]{Cogdell04}. Note that the functional equation in \cite{Cogdell04} is written as $\Lambda(s,\pi)=\varepsilon\Lambda(1-s,\tilde\pi)$, where $\tilde\pi$ is the contragredient representation; one can show that $\Lambda(1-s,\tilde\pi)=\bar\Lambda(1-s,\pi)$ for unitary $\pi$ (recall that $\bar f(s)=\overline{f(\bar s)}$ for a function of a complex variable).

Item \eqref{item:bal} follows by considering the local Langlands parameter at the archimedean place, keeping in mind that the determinant of this parameter corresponds to the central character of $\pi_\infty$.
\end{proof}

The integer $N$ appearing in the functional equation equals $\prod_p p^{a(\pi_p)}$, where $a(\pi_p)$ is the exponent of the conductor of the local representation $\pi_p$. Let $M=\prod_p p^{a(\omega_{\pi,p})}$, where $a(\omega_{\pi,p})$ is the exponent of the conductor of $\omega_{\pi,p}$, the central character of $\pi_p$. By reduction to the supercuspidal case and using the existence of the distinguished vector exhibited in \cite{JacquetPSShalika1981}, one can easily prove that $a(\omega_{\pi,p})\leq a(\pi_p)$. Consequently, $M|N$. We may therefore consider $\chi$, which originally is a Dirichlet character mod $M$, as a Dirichlet character mod $N$.  This character is the character required by Axiom 5).

By definition, $L(s,\pi)$ is an Euler product
\begin{equation}\label{eqn:EP2}
 L(s,\pi)= \prod_{p \, {\rm prime}} F_p(p^{-s})^{-1},
\end{equation}
where each $F_p$ is a polynomial of degree at most $d$ (as required by Axiom 3b)), with $F_p(0)=1$.
Considering Langlands parameters at the non-archimedean places, it
follows
from \eqref{eqn:omegap} that
\begin{equation}\label{eqn:Fpchi2}
 F_p(z)=1+ \cdots + (-1)^d\chi(p) z^d;
\end{equation}
as required by Axioms 3a) and 5a).

Considering Langlands parameters at the archimedean place, it follows 
that
\begin{equation}\label{ccexpleq1}
  \omega_\pi(x)=x^{\mathrm{Im}\left(\sum \mu_j +  \sum(2\nu_k+1)\right)}\qquad\text{for }x\in\R_{>0},
\end{equation}
and so, from \eqref{eqn:omegainfty}, it follows that
\begin{equation}\label{eqn:parityaxiom2}
 \chi(-1) = (-1)^{\sum \mu_j +  \sum(2\nu_j+1)},
\end{equation}
showing that $\Q$-automorphic $L$-functions satisfy Axioms 5b) and 5c).

Conjecturally, each local component $\pi_p$ of a cuspidal, automorphic representation $\pi$ as in Definition \ref{def:automorphic} is \emph{tempered}; see \cite{Sarnak05}, or Conjecture 1.6 of \cite{Clozel1990}. The following lemma lists some consequences of temperedness for the spectral parameters and Satake parameters.

\begin{lem}\label{lem:tempered}
Assume that $\pi=\otimes\pi_p$ is a cuspidal automorphic representation of
$\GL(d,\A_\Q)$. Let $\mu_j,\nu_j$ be as in Theorem \ref{thm:niceness}. Let the polynomial $F_p(z)$ in \eqref{eqn:Fpchi2} be factored as
 \begin{equation}\label{lem:temperedeq1}
  F_p(z) = (1-\alpha_{1,p} z)\cdots (1-\alpha_{d_p,p} z)
 \end{equation}
 with $0\leq d_p\leq d$ and $\alpha_{j,p}\in\C$.
 \begin{enumerate}
  \item Assume that $\pi_\infty$ is tempered. Then for every $j$ we have $\Re(\mu_j) \in \{0,1\}$ and $\Re(\nu_j) \in \{\frac12,1,\frac32,2,...\}$.
  \item Assume that $\pi_p$ is tempered for $p<\infty$ with $p\nmid N$. Then $|\alpha_{j,p}|=1$ for all $j\in\{1,\ldots,d\}$.
  \item Assume that $\pi_p$ is tempered for $p<\infty$ with $p|N$. Then $|\alpha_{j,p}|= p^{-m_j/2}$ for some $m_j\in \{0,1,2,...\}$, and  $\sum m_j \le d-d_p$.
 \end{enumerate}
\end{lem}

This Lemma implies that $\Q$-automorphic $L$-functions satisfy Axiom~4):
the first item in the statement is about Axiom~4a) and the second two are about Axiom~4b).

\begin{proof}
(1) follows from the fact that a representation of $\GL(d,\R)$ is tempered if and only if its Langlands parameter has bounded image. For (2), see \cite{Sarnak05} (also, (2) is a special case of (3)). For (3), we consider the local parameter of $\pi_p$, which is an admissible homomorphism $\varphi:W'(\bar\Q_p/\Q_p)\to\GL(d,\C)$. Here, $W'(\bar\Q_p/\Q_p)$ is the Weil-Deligne group; see \cite{Rorlich1994} for precise definitions, and \cite{Kudla1994} for properties of the local Langlands correspondence. Let ${\rm sp}(n)$ be the $n$-dimensional indecomposable representation of $W'(\bar\Q_p/\Q_p)$ defined in \S5 of \cite{Rorlich1994}. We can write
\begin{equation}\label{lem:temperedeq2}
 \varphi=\bigoplus_{j=1}^t\rho_j\otimes{\rm sp}(n_j),
\end{equation}
with uniquely determined irreducible representations $\rho_j$ of the Weil group $W(\bar\Q_p/\Q_p)$, and uniquely determined positive integers $n_i$. Evidently, $d=\sum_{j=1}^t\dim(\rho_j)n_j$. The Euler factor $L(s,\pi_p)$ is equal to the L-factor of $\varphi$, as defined in \S8 of \cite{Rorlich1994}. By the Proposition in \S8 of \cite{Rorlich1994},
\begin{equation}\label{lem:temperedeq3}
 L(s,\pi)=\prod_{j=1}^tL(s+n_j-1,\rho_j).
\end{equation}
Assume that $\rho_1,\ldots,\rho_r$ are unramified characters of $W(\bar\Q_p/\Q_p)$, and that $\rho_{r+1},\ldots,\rho_t$ are either ramified characters or of dimension greater than $1$. Then
\begin{equation}\label{lem:temperedeq4}
 L(s,\pi)=\prod_{j=1}^rL(s+n_j-1,\rho_j)=\prod_{j=1}^r\frac1{1-\rho_j(p)p^{-s-n_j+1}},
\end{equation}
where we identified $\rho_1,\ldots\rho_r$ with characters of $\Q_p^\times$. Comparison with \eqref{lem:temperedeq1} shows that $d_p=r$ and, after an appropriate permutation,
\begin{equation}\label{lem:temperedeq5}
 \alpha_{j,p}=\rho_j(p)p^{-n_j+1}\qquad\text{for }j=1,\ldots,r.
\end{equation}
Now $\pi_p$ is tempered if and only if the representations $|\cdot|_p^{(n_i-1)/2}\otimes\rho_i$ are bounded for $i=1,\ldots,t$; see \textsection~2.2 of \cite{Kudla1994}. In particular, assuming $\pi_p$ is tempered, we have $|\rho_j(p)|=p^{(n_j-1)/2}$, and thus $|\alpha_{j,p}|=p^{-(n_j-1)/2}$. Setting $m_j=n_j-1$, we have $|\alpha_{j,p}|=p^{-m_j/2}$ and
$$
 \sum_{j=1}^{d_p}m_j=\sum_{j=1}^{d_p}n_j-d_p\leq d-d_p.
$$
This concludes the proof.
\end{proof}

\begin{prop}\label{prop:automorphicanalytic}
 Assume that $\pi=\otimes\pi_p$ is a cuspidal automorphic representation of $\GL(d,\A_\Q)$ such that each local component $\pi_p$ is tempered. Then $L(s,\pi)$ is a tempered analytic $L$-function in the sense of Section~\ref{sec:defn_analytic}.
If $\pi$ is balanced, then $L(s,\pi)$ is balanced.
\end{prop}
\begin{proof}
In the balanced case, this follows from Theorem \ref{thm:niceness}, equations \eqref{eqn:Fpchi2} and \eqref{eqn:parityaxiom2}, and Lemma \ref{lem:tempered}. Assume $\pi$ is not balanced. Then there exists $y\in\R$ such that $|\cdot|^{iy}\otimes\pi$ is balanced. Hence $L(s,|\cdot|^{iy}\otimes\pi)=L(s+iy,\pi)$ is a balanced tempered analytic $L$-function. Consequently, $L(s,\pi)$ is a tempered analytic $L$-function.
\end{proof}

\section{Algebraic and arithmetic \texorpdfstring{$L$}{}-functions}\label{sec:aa}

The most widely studied $L$-functions are those arising from arithmetic
objects such as elliptic and higher-genus curves, holomorphic modular forms,
number fields, Artin representations,
Galois representations, and motives.   We give two characterizations of
such $L$-functions: one in terms of their Dirichlet coefficients, and the
other in terms of their spectral parameters.


\subsection{Analytic \texorpdfstring{$L$}{}-functions of algebraic type}\label{sec:analytic-of-alg-type}

Building on ideas of Stark and Hejhal, Booker, Str\"ombergsson,
and Venkatesh~\cite{bsv}, carried out computations that support the
conjecture that if $\lambda>\frac14$ is the Laplacian eigenvalue of a Maass
form, then $\lambda$ is transcendental.  Since the $\Gamma$-shifts
in the associated $L$-function have imaginary part $\pm \sqrt{\lambda - \frac14}$,
one expects that the imaginary part of any
$\Gamma$-shift in a primitive balanced analytic $L$-function is either 0 or
transcendental.  This motivates the following definition.
\begin{defn}\label{defn:Lalgebraictype}  Suppose $L(s)$ is an analytic $L$-function with
spectral parameters $\{\mu_j\}$ and $\{\nu_k\}$.  We say that $L(s)$
is \term{of algebraic type} if either
every $\mu_j$ and $\nu_k$ is in $\Z$, or
every $\mu_j$ and $\nu_k$ is in $\frac12 + \Z$.
The integer $w_{alg} = 2 \max\{0, \nu_1,\ldots,\nu_{d_2}\}$ is
called the \term{algebraic weight} of the $L$-function.
\end{defn}

The second option: $\mu_j$ and $\nu_k$ are in $\frac12 + \Z$,
implies that there are no $\mu_j$, 
because if $L$ is tempered then $\mu_j\in \{0,1\}$,
and in general (see the Appendix on the unitary pairing condition),
$\mu_j \in (-\frac12,\frac12) \cup (\frac12,\frac32)$.

In Definition~1.6 of \cite{Clozel1990}, Clozel defined the notion of \emph{algebraic}
automorphic representation of $\GL(n)$ over a number field.  We use the term
\emph{of algebraic type} because our notion applies to the $L$-function, not
its underlying representation.  See Section~\ref{sec:Q-analytic-automorphic}
for more details.

The term \emph{algebraic weight} was chosen because if $M$ is a motive
of weight $w$, 
then by Serre's recipe \cite{Serre1969} $L(s,M)$ will have algebraic weight~$w$.

%

\subsection{Analytic \texorpdfstring{$L$}{}-functions of arithmetic type}
\label{sec:defn_arithmetic}

The computations of Booker, Str\"ombergsson,
and Venkatesh~\cite{bsv}
also support the conjecture that in general the Fourier coefficients of
Maass forms with Laplacian eigenvalue $\lambda > \frac14$ are
transcendental and algebraically independent except for the constraints
imposed by the Hecke relations.  Thus we have a complement to the previous
definition, involving a condition on the Dirichlet coefficients.

\begin{defn}\label{defn:Larithmetictype}  Suppose $L(s) = \sum a_n n^{-s}$ is an analytic
$L$-function.  We say that $L(s)$ is \term{of arithmetic type}
if there exists $w_{ar} \in \Z$ 
and a number field $F$ such that
$a_n n^{{w_{ar}}/2}\in \mathcal{O}_F$ for all $n$.
The smallest such $F$ is called the
\term{field of coefficients}, and the smallest such $w_{ar}$ 
is called the \term{arithmetic weight} of the $L$-function.
\end{defn}


An analytic $L$-function with algebraic coefficients is not
necessarily of arithmetic type, as shown by the example
$L(s) = L(s,\chi)L(s,E)$, where $\chi$ is a
primitive Dirichlet character and $E/\Q$ is an elliptic curve.
As indicated in Figure~\ref{fig:diagram2}, it is conjectured that
such examples must be non-primitive.

As we will explain in Section~\ref{sec:connections}, by combining existing
conjectures one obtains the conjecture that a primitive balanced analytic $L$-function
is of algebraic type if and only if it is of arithmetic type.
Furthermore, we have the \term{Hodge conjecture}:  $w_{alg} = w_{ar}$.

\subsection{\texorpdfstring{$\Q$}{}-automorphic \texorpdfstring{$L$}{}-functions of algebraic type}\label{sec:Q-analytic-automorphic}
Let $\A_F$ be the ring of adeles of a number field $F$. In \cite{Clozel1990}, Clozel considered isobaric automorphic representations of $\GL(n,\A_F)$. He called such a representation \term{algebraic} if the local Langlands parameters at all archime\-dean places satisfy certain integrality conditions. More generally, for a connected, reductive $F$-group $G$ and an automorphic representation of $G(\A_F)$, Buzzard and Gee in \cite{BuzzardGee2014} defined the notions of \term{$C$-algebraic} and \term{$L$-algebraic}. If $G=\GL(n)$ and $\pi$ is isobaric, then
$$
 \pi\text{ is algebraic }\;\Longleftrightarrow\;\pi\text{ is $C$-algebraic }\;\Longleftrightarrow\;\pi|\cdot|^{\frac{n-1}2}\text{ is $L$-algebraic}.
$$
In the case of a tempered automorphic representation $\pi\cong\otimes\pi_p$ of $\GL(d,\A_\Q)$, the notions of $C$-algebraic and $L$-algebraic can easily be expressed in terms of the archime\-dean Euler factor. Recall that this factor is of the general form
\begin{equation}\label{archLgeneraleq}
 L(s,\pi_\infty)=\prod_{j=1}^{d_1}\Gamma_\R(s+\mu_j)\prod_{k=1}^{d_2}\Gamma_\C(s+\nu_k)
\end{equation}
with complex numbers $\mu_j$ and $\nu_k$, and $d=d_1+2d_2$.

\begin{lem}\label{LCalglemma}
 Let $\pi\cong\otimes\pi_p$ be an automorphic representation of\linebreak $\GL(d,\A_\Q)$ such that $\pi_\infty$ is tempered. Let $L(s,\pi_\infty)$ be as in \eqref{archLgeneraleq}.
 \begin{enumerate}
  \item Assume that $d$ is even. Then
    \begin{itemize}
     \item $\pi$ is $C$-algebraic if and only if $d_1=0$ and $\nu_k\in\frac12+\Z_{\geq0}$ for $k=1,\ldots,d/2$.
     \item $\pi$ is $L$-algebraic if and only if $\mu_j\in\{0,1\}$ for $j=1,\ldots,d_1$ and $\nu_k\in\Z_{>0}$ for $k=1,\ldots,d_2$.
    \end{itemize}
  \item Assume that $d$ is odd. Then $\pi$ is $C$-algebraic if and only if $\pi$ is $L$-algebraic if and only if  $\mu_j\in\{0,1\}$ for $j=1,\ldots,d_1$ and $\nu_k\in\Z_{>0}$ for $k=1,\ldots,d_2$.
 \end{enumerate}
\end{lem}
\begin{proof}
This follows in a straightforward manner from Definitions~5.7 and~5.9 of \cite{BuzzardGee2014}, and the well-known recipe of attaching $\Gamma$-factors to representations of the real Weil group.
\end{proof}

\begin{rmk} Suppose $L(s,\pi)$ is a $\Q$-automorphic $L$-function, coming from a unitary cuspidal automorphic representation $\pi=\otimes_{p\leq\infty}\pi_p$ of $\GL(d,\A_\Q)$, as in Definition \ref{def:automorphic}.  In particular, as shown in the course of Section~\ref{sec:Q-auto}, $L(s,\pi)$ is an analytic $L$-function.
Thus, we can say that $L(s,\pi)$
is \term{of algebraic type} if it satisfies the conditions of
Definition~\ref{defn:Lalgebraictype}.
\end{rmk}

Therefore by the lemma we have:

%
%

\begin{rmk}\label{rem:Q1cupQ2}
By the Ramanujan conjecture, all local components $\pi_p$ of the cuspidal automorphic representation $\pi$ are tempered. Assuming this is the case, we see that $L(s,\pi)$ is of algebraic type if and only if $\pi$ is either $C$-algebraic or $L$-algebraic.
\end{rmk}


\subsection{\texorpdfstring{$\Q$}{}-automorphic \texorpdfstring{$L$}{}-functions of arithmetic type}\label{ssec:arithmetic}
Let $G$ be a connected, reductive group over the number field $F$, and let $\pi\cong\otimes\pi_v$ be an automorphic representation of $G(\A_F)$. Buzzard and Gee \cite{BuzzardGee2014} define the notions of $\pi$ being \term{$C$-arithmetic} and \term{$L$-arithmetic} in terms of the Satake parameters of $\pi_v$ at almost all places. For $G=\GL(n)$ it is true that
$$
 \pi\text{ is $C$-arithmetic }\;\Longleftrightarrow\;\pi|\cdot|^{\frac{n-1}2}\text{ is $L$-arithmetic}.
$$
The conditions can easily be reformulated in terms of $L$-functions:

\begin{lem}\label{LCarithlemma}
 Let $\pi\cong\otimes\pi_p$ be an automorphic representation of \linebreak$\GL(d,\A_\Q)$. Let $S$ be a finite set of primes such that $\pi_p$ is unramified for primes $p\notin S$. Let
 \begin{equation}\label{LCarithlemmaeq1}
  L(s,\pi_p)=\frac1{(1-\alpha_{p,1}p^{-s})\cdot\ldots\cdots(1-\alpha_{p,d}p^{-s})}
 \end{equation}
 be the Euler factor for $p\notin S$.
 \begin{enumerate}
  \item Assume that $d$ is even. Then
    \begin{itemize}
     \item $\pi$ is $C$-arithmetic if and only if there exists a number field $E$ such that $\alpha_{p,1}\sqrt{p},\ldots,$ $\alpha_{p,d}\sqrt{p}\in E$ for almost all $p\notin S$.
     \item $\pi$ is $L$-arithmetic if and only if there exists a number field $E$ such that $\alpha_{p,1},\ldots,$ $\alpha_{p,d}\in E$ for almost all $p\notin S$.
    \end{itemize}
  \item Assume that $d$ is odd. Then $\pi$ is $C$-arithmetic if and only if $\pi$ is $L$-arithmetic if and only if there exists a number field $E$ such that $\alpha_{p,1},\ldots,\alpha_{p,d}\in E$ for almost all $p\notin S$.
 \end{enumerate}
\end{lem}

The following is equivalent to Definition~\ref{defn:Larithmetictype}, but we include this formulation because
it is stated in terms of the parameters which are more natural for 
$\Q$-automorphic $L$-functions.

\begin{defn}\label{def:automorphic-arith}
 Let $L(s,\pi)$ be a $\Q$-automorphic $L$-function, coming from a cuspidal automorphic representation $\pi=\otimes_{p\leq\infty}\pi_p$ of $\GL(d,\A_\Q)$, as in Definition \ref{def:automorphic}.  Let
 \begin{equation}\label{def:automorphic-aritheq1}
  L(s,\pi_p)=\frac1{(1-\alpha_{p,1}p^{-s})\cdot\ldots\cdots(1-\alpha_{p,d}p^{-s})}
 \end{equation}
 be the Euler factor at a prime $p$. We say that $L(s,\pi)$ is \term{of arithmetic type} if there exists a number field $E$ such that either
 \begin{equation}\label{def:automorphic-aritheq2}
  \alpha_{p,1},\ldots,\alpha_{p,d}\in E\qquad\text{for almost all good primes $p$}
 \end{equation}
 or
 \begin{equation}\label{def:automorphic-aritheq3}
  \alpha_{p,1}\sqrt{p},\ldots,\alpha_{p,d}\sqrt{p}\in E\qquad\text{for almost all good primes $p$}.
 \end{equation}
\end{defn}

\begin{rmk}
Let $\pi$ be as in Definition \ref{def:automorphic-arith} and suppose $d$ is odd. 
Then \eqref{def:automorphic-aritheq3} cannot occur.  
 Indeed, if \eqref{def:automorphic-aritheq3} would hold, then we would also have
\begin{equation}\label{def:automorphic-aritheq4}
  \alpha_{p,1}\cdot\ldots\cdot\alpha_{p,d}\sqrt{p}\in E\qquad\text{for almost all good primes $p$}.
\end{equation}
But the numbers $\alpha_{p,1}\cdot\ldots\cdot\alpha_{p,d}$ are the Satake parameters of the central character $\omega_\pi$ of $\pi$, which, up to a unitary twist, corresponds to a Dirichlet character $\chi$; see \eqref{eqn:omegap}. Hence we would have
\begin{equation}\label{def:automorphic-aritheq5}
  \chi(p)p^{it}\sqrt{p} = \chi(p)p^{it + \frac12}\in E\qquad\text{for almost all good primes $p$}
\end{equation}
for some real number $t$, so in particular
\begin{equation}\label{def:automorphic-aritheq6}
p^{it + \frac12}\in E\qquad\text{for infinitely many $p$. }
\end{equation} 
This is impossible for $t\in \R$ because of the following consequence of
the Six Exponentials Theorem~\cite{MR972013}.
\end{rmk}

\begin{lem} If $\alpha\in \C$ and $E/\Q$ is a number field with 
$p^\alpha \in E$ for infinitely many primes~$p$, then $\alpha\in \Z$.
\end{lem}

\begin{proof}  Suppose $\alpha \not\in \Z$. We cannot have $\alpha\in \Q$
because $E$ is a finite extension of $\Q$.
Thus, $\{1, \alpha\}$ is linearly independent over the rationals. 
(We also cannot have $\alpha$ algebraic, because by the Gelfond-Schneider
theorem~$p^\alpha$ would be transcedental. This does not seem to be
needed in the proof.)

Suppose $p_1$, $p_2$, and $p_3$ are distinct primes with $p_j^\alpha \in E$.
Since $\{\log p_1, \log p_2, \log p_3\}$ is linearly independent over the
rationals we have a contradiction because by the Six Exponentials Theorem~\cite{MR972013},
one of $p_1$, $p_2$, $p_3$, $p_1^\alpha$, $p_2^\alpha$, and $p_3^\alpha$ must
be transcendental.
\end{proof}

As a consequence of Lemma \ref{LCarithlemma} and the above remark,
we see that $L(s,\pi)$ is of arithmetic type if and only if $\pi$ is either $C$-arithmetic or $L$-arithmetic (see Section~\ref{sec:connections}).
So the conjectures of Clozel and Buzzard-Gee yield:

\begin{conj}\label{conj:Q1=Q2}
 A $\Q$-automorphic $L$-function is of algebraic type if and only if it is of arithmetic type.
\end{conj}
A more general conjecture for automorphic representations $\pi$ of \linebreak $G(\A_F)$ was made
by Buzzard and Gee in \cite{BuzzardGee2014}.
Namely, $\pi$ is $L$-algebraic if and only if it is $L$-arithmetic, and $\pi$ is $C$-algebraic if and only if it is $C$-arithmetic. The surprising fact about these conjectures is that a condition purely in terms of archimedean $L$-parameters is conjecturally equivalent to a condition purely in terms of non-archimedean $L$-parameters.

For isobaric automorphic representations of $\GL(n,\A_F)$, the conjecture that $C$-algebraic implies $C$-arithmetic is also a consequence of the more general conjectures made in \cite{Clozel1990, MR3642468}; see \textsection~8.1 of \cite{BuzzardGee2014}.

\subsection{\texorpdfstring{$L$}{}-functions of motives}\label{ssec:motives}
Let $F$ be a number field. The category of motives over $F$ has been constructed by Grothendieck in the 1960s. The first article about motives seems to be \cite{Demazure1969}. More contemporary and useful surveys are \cite{Kleiman1994}, \cite{Scholl1994} and \cite{Milne2013}. We will not recall here the construction of the category of motives. What is important for us is that given a \emph{pure motive $M$ of weight $w$}, there is an $L$-function $L(M,s)$ attached to it, which, after completing it to a function $\Lambda(M,s)$ using appropriate $\Gamma$-factors, conjecturally satisfies a functional equation $\Lambda(M,s)=\pm\bar\Lambda(M,1+w-s)$. We write $\Lambda(M,s)$ and not $\Lambda(s,M)$, because the functional equation relates $s$ and $1+w-s$, and not $s$ and $1-s$.

We briefly recall the shape of $\Lambda(M,s)$, using \cite{Serre1969} as our reference. We assume that the ground field is $\Q$ for simplicity. For each non-archimedean place $p$ the characteristic polynomial $P_{w,p}$ of the action of Frobenius on the inertia-fixed points of the $w$-th \'{e}tale cohomology group (at least conjecturally) has coefficients in $\Z$. We set
\begin{equation}\label{motLeq3}
 L_p(M,s)=\frac1{P_{w,p}(p^{-s})},
\end{equation}
and $L(M,s)=\prod_{p<\infty}L_p(M,s)$. There exists a positive integer $d$, called the \term{rank} of $M$, such that for almost all places the polynomial $P_{w,p}$ will have degree $d$. For all ``good'' places $p$, if we factor
\begin{equation}\label{motLeq4}
 P_{w,p}(T)=\prod_{j=1}^{b_w}(1-\alpha_{j,p}T),\qquad\alpha_{j,p}\in\C^\times,
\end{equation}
then it is conjectured that
\begin{equation}\label{motLeq5}
 |\alpha_{j,p}|=p^{w/2}
\end{equation}
(see $C_7$ in \textsection~2.3 of \cite{Serre1969}). At the archimedean place $\infty$ we have
\begin{equation}\label{motLeq1}
 L_\infty(M,s)=\Gamma_\R\Big(s-\frac w2\Big)^{h^{w/2,+}}\Gamma_\R\Big(s-\frac w2+1\Big)^{h^{w/2,-}}\prod_{\substack{p+q=w\\p<q}}\Gamma_\C(s-p)^{h^{p,q}},
\end{equation}
with the first two factors only appearing if $w$ is even. Here, the $h^{p,q}$ are the dimensions of the spaces $H^{p,q}$ in the Hodge decomposition of the Betti realization of $M$. If $w$ is even, then there is a space $H^{p,p}$ ($p=w/2$), on which complex conjugation acts as an involution; the number $h^{p,\pm}$ is the dimension of the $\pm1$-eigenspace.

Conjecturally, there exists a positive integer $N$ such that
\begin{equation}\label{motLeq2a}
 \Lambda(M,s)=N^{(s-w/2)/2}L_\infty(M,s)L(M,s)
\end{equation}
extends to an entire function on all of $\C$, bounded in vertical strips, and satisfies the functional equation $\Lambda(M,s)=\pm\bar\Lambda(M,1+w-s)$.

Replacing $s$ by $s+w/2$, we obtain the \emph{analytically normalized} functions $L(s,M):=L(M,s+w/2)$ and $\Lambda(s,M):=\Lambda(M,s+w/2)$. The functional equation becomes $\Lambda(s,M)=\pm\bar\Lambda(1-s,M)$. The factor \eqref{motLeq1} turns into
\begin{equation}\label{motLeq2}
 L_\infty(M,s+w/2)=\Gamma_\R(s)^{h^{w/2,+}}\Gamma_\R(s+1)^{h^{w/2,-}}\prod_{\substack{p+q=w\\p<q}}\Gamma_\C\Big(s+\frac{q-p}2\Big)^{h^{p,q}}.
\end{equation}
By \eqref{motLeq5}, the roots of the denominator polynomials of $L(s,M)$ will have absolute value $1$.

As described in \textsection~4.3 of \cite{Clozel1990}:
\begin{conj}\label{motautconj}
 There exists a one-to-one correspondence between irreducible, pure motives $M$ over $\Q$ of rank $d$, and $C$-algebraic cuspidal automorphic representations $\pi$ of $\GL(d,\A_\Q)$, such that $L(s,M)=L(s,\pi)$. 
\end{conj}
The conjecture implies that the class of analytically normalized $L$-functions arising from irreducible, pure motives is the same as the class of $\Q$-automorphic $L$-functions of algebraic type. 
The conjecture that $L(s,M)$ should satisfy the required analytic properties
shared by the class of analytic $L$-functions as defined in Section~\ref{sec:defn_analytic} is known as the \term{Hasse-Weil conjecture}.

\begin{rmk} Given a $\Gamma$-factor in the analytic normalization \eqref{motLeq2},
it is not possible to determine the weight of the underlying motive.
Indeed, the motives which (conjecturally) are attached to the $L$-function
form an equivalence class, where the members
are Tate twists of each other.
It is natural to choose a twist so that non-vanishing Hodge numbers
of that motive are among
$(h^{w,0},...,h^{0,w})$, with
$h^{w,0}=h^{0,w}>0$.  
The weight, $w$, of that motive will equal the algebraic weight of the
$L$-function, which explains our choice of terminology.
\end{rmk}

\subsection{\texorpdfstring{$L$}{}-functions of Galois representations}\label{ssec:Galois}
Let ${\rm Gal}(\bar\Q/\Q)$ be the absolute Galois group of $\Q$. Let $L$ be a finite extension of $\Q_\ell$ for some prime $\ell$.  A continuous homomorphism $\rho:{\rm Gal}(\bar\Q/\Q)\to\GL(d,L)$ will be referred to as a \term{Galois representation}. See \cite{Serre1968} for basic facts. In \cite{FontaineMazur1995} the class of \term{geometric} Galois representations was defined. As summarized in Taylor \cite{Taylor}, Galois representations arising from motives are geometric. Conversely, the Fontaine-Mazur conjecture asserts that any geometric Galois representation is motivic. Hence, the class of (analytically normalized) $L$-functions arising from geometric Galois representations should be the same as the class of (analytically normalized) $L$-functions attached to Galois representations. Assuming Conjecture \ref{motautconj}, this is also the same as the class of $\Q$-automorphic $L$-functions of algebraic type.
This explains the triangle in the upper left corner of Figure~\ref{fig:diagram2}.

Conjecture 5.16 of \cite{BuzzardGee2014}, for $G=\GL(d)$ over $\Q$, makes the statement that a geometric Galois representation is attached to an $L$-algebraic automorphic representation $\pi$ of $\GL(d,\A_\Q)$. For recent progress on this conjecture, see \cite{Shin2011} and \cite{HarrisLanTaylorThorne2016}. That every Galois representation arises from an automorphic representation is known as the Modularity Conjecture.

\section{Connections}\label{sec:connections}

The equalities between sets of $L$-functions in Figure~\ref{fig:diagram1} are a consequence of the 12 relations shown in Figure~\ref{fig:diagram2}, where an arrow means ``can be viewed as a natural subset of.''  Most of those arrows are at least partially conjectural.  More detailed explanations of these arrows can be found in Tables~\ref{tbl:table1} and \ref{tbl:table2}.
\begin{table}
\def\arraystretch{1.5}
\begin{tabular}{lll}
Connection & Label & Justification\\\hline
$Q \subset A$ & J-PS-S & \begin{minipage}[t]{.5\columnwidth}A result generally due to Cogdell--Piatetski-Shapiro and Jacquet--Piatetski-Shapiro--Shalika \cite{JacquetPSShalika1981,Cogdell04} as formulated in Proposition~\ref{prop:automorphicanalytic}.\end{minipage}\\
$A\subset Q$ & S & \begin{minipage}[t]{.5\columnwidth}Selberg \cite{sel} identified a class of axioms for $L$-functions and implicitly conjectured that this class is contained in the class associated to automorphic representations as defined by Langlands \cite{Langlands1980}.  $A$ and $Q$ are, respectively, subsets of these classes with the same formal properties.\end{minipage}\\
$Q_*\subset Q$ & &  \begin{minipage}[t]{.5\columnwidth}Restriction to a subset.\end{minipage}\\
$Q_1\cup Q_2 = Q_*$ & & \begin{minipage}[t]{.5\columnwidth} Remark~\ref{rem:Q1cupQ2}.\end{minipage}\\
$Q_1=Q_2$ & B-G &  \begin{minipage}[t]{.5\columnwidth}A conjecture due to Buzzard--Gee \cite{BuzzardGee2014} as formulated in Conjecture~\ref{conj:Q1=Q2}.\end{minipage}\\
$A_*\subset A$ & & \begin{minipage}[t]{.5\columnwidth}Restriction to a subset.\end{minipage}\\
$Q_1 = A_1$ & & \begin{minipage}[t]{.5\columnwidth}Formal if we assume $A=Q$.\end{minipage}\\
$Q_2 \subset A_2$ & & \begin{minipage}[t]{.5\columnwidth} A result of Clozel \cite{Clozel1990} implies that the (suitably) re-scaled coefficients are integers and the inclusion otherwise is formal.\end{minipage}\\
$A_2\subset Q_2$ &&\begin{minipage}[t]{.5\columnwidth}Formal.\end{minipage}\\
$A_1=A_2$ & & \begin{minipage}[t]{.5\columnwidth} Piecing together previous connections.\end{minipage}\\
\begin{minipage}[t]{.3\columnwidth}
$A_*=A_1=A_2\\ \phantom{A_*}=Q_1=Q_2=Q_*$\end{minipage}&&\begin{minipage}[t]{.5\columnwidth}Piecing together previous connections.\end{minipage}\\
&&\\
\end{tabular}
\caption{Explanations of the arrows in Figure~\ref{fig:diagram2} between analytic $L$-functions and $\Q$-automorphic $L$-functions.  These explanations and these arrows correspond to Section~\ref{sec:Q-auto} and Sections~\ref{sec:analytic-of-alg-type}--\ref{ssec:arithmetic}. If the justification for a connection is ``formal'', this means that it is an immediate corollary of the axioms or properties satisfied by the sets of $L$-functions being connected.}
\label{tbl:table1}
\end{table}

\begin{table}
\def\arraystretch{1.5}
\begin{tabular}{lll}
Connection & Label & Justification\\\hline
$M\subset A_*$ & H-W & \begin{minipage}[t]{.5\columnwidth}In our notation this is a restatement of the Hasse--Weil conjecture.\end{minipage}\\
$M=Q_*$ & C & \begin{minipage}[t]{.5\columnwidth}Conjecture due to Clozel\cite{Clozel1990, MR3642468}.\end{minipage}\\
$M\subset G$ & T & \begin{minipage}[t]{.5\columnwidth}Taylor \cite[pp, 77, 79-80]{Taylor} distills the work of many people and describes how to attach a Galois representation to a motive.\end{minipage}\\
$G\subset M$ & F-M & \begin{minipage}[t]{.5\columnwidth}Fontaine-Mazur Conjecture \cite{FontaineMazur1995}.\end{minipage}\\
$G\subset Q_*$ & Modularity & \begin{minipage}[t]{.5\columnwidth}The general conjecture that asserts to a Galois representation one can attach an automorphic representation so that the two $L$-functions agree.\end{minipage}\\
$Q_*\subset G$ & B-G & \begin{minipage}[t]{.5\columnwidth}Conjecture due to Buzzard--Gee \cite{BuzzardGee2014}.\end{minipage}\\
$M=A_*=Q_*=G\!$ &&\begin{minipage}[t]{.5\columnwidth}Piecing together previous connections.\end{minipage}
\end{tabular}
\caption{Explanations of arrows in Figure~\ref{fig:diagram2} between various sources of $L$-functions as described in Sections~\ref{ssec:motives} and \ref{ssec:Galois}.}\label{tbl:table2}
\end{table}











\begin{landscape}
\usetikzlibrary{arrows,
                calc,chains,
                decorations.pathreplacing,decorations.pathmorphing,
                fit,
                positioning,}


\vspace{5in}

\begin{figure}\centering
\scalebox{.65}{
\begin{tikzpicture}[node distance = 8mm, start chain = A going below, font = \sffamily, > = stealth', PC/.style = {
        rectangle, rounded corners,
        draw=black, very thick,
        text width=15em,
        minimum height=3em,
        align=center,
        on chain}, 
blank/.style={text width=10em,
        minimum height=3em,
        align=center, on chain},
                    ]

\node[PC]   {$\mathbb{Q}$-automorphic $L$-functions of algebraic type}; 
\node[PC]   {$\mathbb{Q}$-automorphic $L$-functions of arithmetic type}; 
\node[blank] {};
\node[PC]   {Analytic $L$-functions of algebraic type};
\node[PC]   {Analytic $L$-functions of arithmetic type};
%

%
\node[PC, inner xsep=3mm, inner ysep=6mm, yshift=9mm,
      fit=(A-1) (A-2)] {};

\node[PC, inner xsep=3mm, inner ysep=6mm, yshift=9mm,
      fit=(A-4) (A-5)] {};

\path   let \p1 = (A-6.north),
            \p2 = (A-6.south),
            \n1 = {veclen(\y2-\y1,\x2-\x1)} in
node[PC, right=of A-6.east,xshift=10mm,minimum height=\n1]  {$\mathbb{Q}$-automorphic:\\ $L$-functions of tempered balanced unitary cuspidal automorphic representations of $\textrm{GL}(d,\mathbb{A}_{\mathbb{Q}})$};

\path   let \p1 = (A-7.north),
            \p2 = (A-7.south),
            \n1 = {veclen(\y2-\y1,\x2-\x1)} in
node[PC, right=of A-7.east,xshift=10mm,minimum height=\n1]  {Analytic:\\ tempered balanced primitive entire analytic $L$-functions };

\path   let \p1 = (A-6.north),
            \p2 = (A-6.south),
            \n1 = {veclen(\y2-\y1,\x2-\x1)} in
node[PC, left=of A-6.west,xshift=-15mm,minimum height=\n1]  {$L(s,\rho)$ for $\rho$ an irreducible geometric Galois representation};

\path   let \p1 = (A-7.north),
            \p2 = (A-7.south),
            \n1 = {veclen(\y2-\y1,\x2-\x1)} in
node[PC, left=of A-7.west,xshift=-15mm,minimum height=\n1]  {$L(s,M)$ for $M$ an irreducible pure motive};

\node[above left] at (A-1.south east) {\tiny$\mathbf{Q}_1$};
\node[above left] at (A-2.south east) {\tiny$\mathbf{Q}_2$};
\node[above left] at (A-6.south east) {\tiny$\mathbf{Q}_\star$};

\node[above left] at (A-4.south east) {\tiny$\mathbf{A}_1$};
\node[above left] at (A-5.south east) {\tiny$\mathbf{A}_2$};
\node[above left] at (A-7.south east) {\tiny$\mathbf{A}_\star$};

\node[above left] at (A-8.south east) {\tiny$\mathbf{Q}$};
\node[above left] at (A-9.south east) {\tiny$\mathbf{A}$};

\node[above left] at (A-10.south east) {\tiny$\mathbf{G}$};
\node[above left] at (A-11.south east) {\tiny$\mathbf{M}$};




\draw[black,very thick,->, transform canvas={xshift=-10mm}]
    (A-10) edge node[xshift=-5mm] {F-M} (A-11) ;

\draw[black,very thick,->, transform canvas={xshift=10mm}]
    (A-11) edge node[xshift=2.5mm] {T} (A-10) ;

\draw[black,very thick,->, transform canvas={xshift=0mm}]
    (A-1) edge node[xshift=-5mm] {B-G}(A-2);  

\draw[black,very thick,->, transform canvas={xshift=0mm}]
(A-2) edge (A-1);

\path[black,very thick, <->,shorten <=-1mm,shorten >=-1mm]
  (A-11.north east) edge node[auto] {C} (A-6.south west);

\draw[black,very thick, ->, transform canvas={xshift=0mm}]
  (A-11) edge node[auto] {H-W} (A-7);

\draw[black,very thick, ->, transform canvas={yshift=7mm}]
  (A-10) edge node[auto] {Modularity} (A-6);


\draw[black,very thick, ->, transform canvas={yshift=-7mm}]
  (A-6) edge node[auto] {B-G} (A-10);

\path[black,very thick, ->,transform canvas={yshift=0mm}]
   ($(A-2.west)+(0,0.25)$) edge[out=225,in=135]  ($(A-5.west)+(0,-0.25)$) ;

\path[black,very thick,->, transform canvas={yshift=0mm}]
    ($(A-5.west)+(0,0.25)$) edge[out=135,in=225]  ($(A-2.west)+(0,-0.25)$) ;

\path[black,very thick,<->, transform canvas={xshift=0mm}]
    (A-1.east) edge[out=-45,in=45]  (A-4.east) ;


\draw[black,very thick,right hook->, transform canvas={yshift=0mm}]
    (A-6) edge (A-8) ;

\draw[black,very thick,right hook->, transform canvas={yshift=0mm}]
    (A-7) edge (A-9) ;

\draw[black,very thick,->, transform canvas={xshift=10mm}]
    (A-9) edge node[xshift=3mm] {S} (A-8) ;

\draw[black,very thick,->, transform canvas={xshift=-10mm}]
    (A-8) edge node[xshift=-7mm] {C-PS--J-PS-S}(A-9) ;

  \end{tikzpicture}
}
\caption{More detailed conjectured relationships between the sets of $L$-functions considered in this paper.
The relationships are labeled with an attribution to give an indication of the source of the conjecture.}\label{fig:diagram2}
\end{figure}
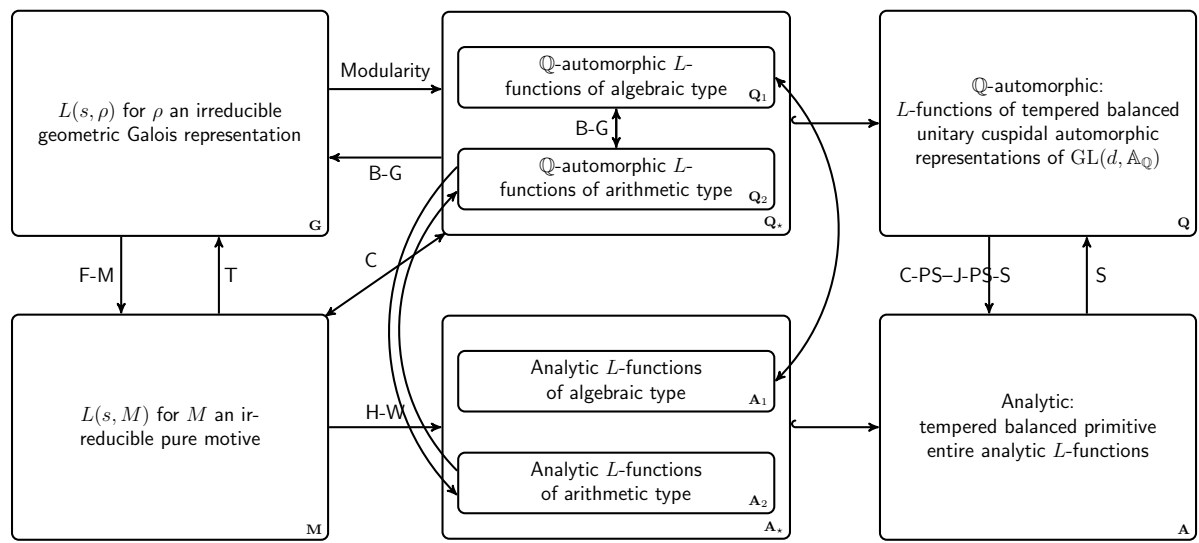

  \end{landscape}

\section{Appendix: non-tempered \texorpdfstring{$L$}{}-functions}

If an $L$-function fails to be tempered, the failure could either
occur in the $\Gamma$-factors or in the local factors of the
Euler product.  The failure must occur in a specific form,
which we call the \term{unitary pairing condition}.
Our motivation is that the unitary pairing condition holds
for the factors arising from generic unitary local representations.

\subsection{The unitary pairing condition at infinity}
In the definition we use the following notation: if $x\in\R$ and
$\xi\in\C$ then $(x,\xi)^*=(x,-\overline{\xi})$.
Also, we introduce a parameter $\theta < \frac12$ which 
measures how far the $L$-function is from being tempered at
infinity.

\begin{defn}
The multisets $\{\mu_j\}$ and $\{\nu_j\}$ meet the
\term{unitary pairing condition at infinity} if it is possible to
write $\mu_j=\delta_j+\alpha_j$ and $\nu_j=\eta_j+\beta_j$,
where $\delta_j\in \{0,1\}$ and $\eta_j\in\{\frac12,1,\frac32,\ldots\}$,
with $|\Re(\alpha_j)|, |\Re(\beta_j)| <\theta$,  such that the multisets
$\{(\delta_j,\alpha_j)\}$ and $\{(\eta_j,\beta_j)\}$ are closed under the operation
$S\to {S}^*$.
\end{defn}

For example, the following $\Gamma$-factor satisfies the unitary
pairing condition:
\begin{align}
&
\Gamma_\R(s-0.2)
\Gamma_\R(s+0.2)
\Gamma_\R(s)^3
\Gamma_\R(s+0.9)
\Gamma_\R(s+1.1)\cr
&\phantom{xxxxx}\times\Gamma_\C(s+0.7)
\Gamma_\C(s+1.3)^2
\Gamma_\C(s+1.7)
\Gamma_\C(s+7),
\end{align}
as does this one
\begin{align}
&
\Gamma_\R(s-0.2+3i)
\Gamma_\R(s+0.2+3i)
\Gamma_\R(s+1)
\Gamma_\R(s+1-8i) \cr
&\phantom{xxxxx}\times
\Gamma_\C(s+0.7)
\Gamma_\C(s+1.3)
\Gamma_\C(s+1.3-7i)
\Gamma_\C(s+1.7-7i).
\end{align}

\subsection{The unitary pairing condition at \texorpdfstring{$p$}{}}
Just as in the archimedean case, we introduce a parameter $\theta<\frac12$
which provides a weak version of the Ramanujan bound:
$|\alpha_{j,p}|\le p^\theta$.
At a good prime the unitary pairing condition is easy to state.

\begin{defn} Suppose $p$ is a good prime.
The multiset $\{\alpha_1,\ldots,\alpha_d\}$
meets the \term{unitary pairing condition at}~$p$
with
partial Ramanujan bound~$\theta<\frac12$,
if $|\alpha_j|\le p^\theta$ and the multiset is
closed under the operation $x\to 1/\overline{x}$.  Equivalently,
the polynomial $F(z)=\prod_j (1-\alpha_j z)$
has all its roots in $|z|\ge p^{-\theta}$ and  satisfies the
self-reciprocal condition
\begin{equation}\label{eqn:self-reciprocal}
F(z) = \xi z^d \overline{F}(z^{-1}),
\end{equation}
where $\xi=(-1)^d \prod_j \alpha_j$.
\end{defn}
The term \emph{self-reciprocal} refers to the fact that,
up to multiplication by a constant,
the coefficients of the polynomial are the same if read in either order.

If $|\alpha_j|=1$ then the unitary pairing condition at $p$ says nothing,
because $\alpha_j=1/\overline{\alpha_j}$.
But those Satake parameters which are not on the unit circle occur in pairs:
if $\alpha_j=r e^{i\theta}$ with $r\not=1$, then $r^{-1}e^{i\theta}$ is also a
Satake parameter.  Those two
points are located symmetrically with respect to the unit circle.

The general case of the unitary pairing condition at $p$, which includes
the good prime version above, closely follows the archimedean case.
Specifically, the $\Gamma_\R$ factors are like the good primes, and the
$\Gamma_\C$ factors are similar to the bad primes.

Recall the notation $(x,\xi)^*=(x,-\overline{\xi})$.

\begin{defn}
The multiset $\{\alpha_1,\ldots,\alpha_M\}$ meets the
\term{unitary pairing condition at}~$p$ of degree $d$
and  partial Ramanujan bound~$\theta<\frac12$,
if it is possible to
write $\alpha_j=p^{-\eta_j-\beta_j}$
where $\eta_j\in\{0,\frac12,1,\frac32,\ldots\}$,
with $\sum 2\eta_j \le d-M$ and
$|\Re(\beta_j)| \le \theta$,  such that the multiset
$S=\{(\eta_j,\beta_j)\}$ is closed under the operation
$S\to {S}^*$.
\end{defn}










\bibliography{analytic-l-fcns}
\bibliographystyle{amsplain}
\end{document}